\documentclass[11pt,a4paper]{article}
\usepackage{amsfonts,color}
\usepackage{amsthm}
\usepackage{amsmath}
\usepackage{amscd}
\usepackage[utf8]{inputenc}
\usepackage{t1enc}
\usepackage[mathscr]{eucal}
\usepackage{indentfirst}
\usepackage{graphicx}
\usepackage{graphics}
\usepackage{pict2e}
\usepackage{epic}
\usepackage{url}
\usepackage{epstopdf}
\usepackage{comment}
\usepackage{hyperref}
\usepackage{tikz}
\usetikzlibrary{positioning}

\newcommand\scalemath[2]{\scalebox{#1}{\mbox{\ensuremath{\displaystyle #2}}}}

\numberwithin{equation}{section}
\usepackage[margin=2.6cm]{geometry}

\theoremstyle{plain}
\newtheorem{Th}{Theorem}[section]
\newtheorem{Lemma}[Th]{Lemma}
\newtheorem{Cor}[Th]{Corollary}
\newtheorem{Prop}[Th]{Proposition}

\theoremstyle{definition}
\newtheorem{Def}[Th]{Definition}
\newtheorem{Conj}[Th]{Conjecture}

\newtheorem{Remark}[Th]{Remark}
\newtheorem{?}[Th]{Problem}

\newcommand{\p}{\mathbb{P}}
\newcommand{\tree}{\mathbb{T}}

\newcommand{\nat}{\mathbb{N}}
\newcommand{\gpr}{\square} 
\newcommand{\e}{\varepsilon}
\newcommand{\eps}{\varepsilon}
\newcommand{\R}{{\mathbb R}}
\newcommand{\Prob}{\mathbb{P}}
\newcommand{\UST}{\mathsf{UST}}
\newcommand{\FSF}{\mathsf{FSF}}
\newcommand{\FSFc}{\FSF_w}
\newcommand{\FSFw}{\FSF_w}
\newcommand{\FUSF}{\mathsf{FUSF}}

\begin{document}

\title{
Connectedness of the Free Uniform Spanning Forest as a function of edge weights}
\author{
Marcell Alexy 
\and 
M\'arton Borb\'enyi 
\and
Andr\'as Imolay
\and
\'Ad\'am Tim\'ar }
\date{}
\maketitle

\begin{abstract}
Let $G$ be the Cartesian product of a regular tree $T$ and a finite connected transitive graph $H$. It is shown in \cite{PT} that the Free Uniform Spanning Forest ($\FSF$) of this graph may not be connected, but the dependence of this connectedness on $H$ remains somewhat mysterious. We study the case when a positive weight $w$ is put on the edges of the $H$-copies in $G$, and conjecture that the connectedness of the $\FSF$ exhibits a phase transition. For large enough $w$ we show that the $\FSF$ is connected, while for a large family of $H$ and $T$, the $\FSF$ is disconnected when $w$ is small (relying on \cite{PT}). Finally, we prove that when $H$ is the graph of one edge, then for any $w$, the $\FSF$ is a single tree, and we give an explicit formula for the distribution of the distance between two points within the tree. 
\end{abstract}

\section{Introduction}
Consider some finite graph $H$ with a weight function (``conductances'') $\hat w: E(H)\to \R^+_0$ on its edges. One may take an unweighted graph and view it as one where the weights are constant 1. 
Choose a spanning tree of $H$ at random, where the probability of a spanning tree $T$ will be proportional to $\Pi_{e\in E(T)} \hat w(e)$. The so-defined probability measure is called the {\it Uniform Spanning Tree} ($\UST$) of $(H,\hat w)$. For a given {\it infinite} graph $G$ and conductances $\hat w$, consider some exhaustion of $G$ by a sequence of connected finite graphs $G_n$, and let $\UST (G_n)$ be the $\UST$ of the weighted graph $(G_n,\hat w|_{G_n})$. It is known
that the weak limit of $\UST (G_n)$ exists, meaning that for any $e_1,...,e_k,f_1,...,f_m\in E(G)$, $\Prob(e_1,...,e_k\in \UST (G_n),\, f_1,..., f_m\not\in\UST (G_n))$ converges, and to the same limit for any choice of the sequence $G_n$. The limiting measure is called the {\it Free Uniform Spanning Forest} ($\FUSF$ or $\FSF$) of $(G,\hat w)$. See \cite{LP} for background, references, and the basic properties of the $\FSF$.
 
It was generally expected that the $\FSF$ of ``tree-like graphs'' would consist of a single tree, until G\'abor Pete and the last author showed in \cite{PT} that for suitably chosen $d$ and connected finite transitive graph $H$, the Cartesian product $\tree^d\gpr H$ of the $d$-regular tree $\tree^d$ and $H$ has a disconnected $\FSF$. From the proof, however, it is not clear what happens to the disconnectedness of the $\FSF$ if we do some natural changes to $d$ and $H$, e.g., increase $d$ with a fixed $H$, or fix $d$ and take a lift of $H$. No monotonicity result of this type is known. A question in the same spirit is to ask how the connectedness of the $\FSF$ changes if we put constant positive weight $w$ on every edge of the $H$-copies in $\tree^d\gpr H$ and then change $w$. What happens if $w$ is very small or large? Does there always exist some $w$ where there is a single $\FSF$ component? Is there always some $w$ where there are infinitely many $\FSF$ components (with an $H$ that has at least 2 vertices)? Is there any kind of monotonicity in $w$, and perhaps even a critical value that separates the phases of disconnectedness and connectedness? The present paper contributes to the understanding of these questions. In particular, initial steps are taken 
in Conjecture \ref{main_conjecture}.

Let $H$ be a finite connected graph and $\tree^d$ be the $d$-regular tree. For an arbitrary given $w>0$, define the weight function $\hat w$ on the edges of $G=\tree^d\gpr H$ so that $\hat w(e)=1$ if $e$ is of the form $\{(x_1,y),(x_2,y)\}$ and $\hat w(e)=w$ if $e$ is of the form $\{(x,y_1),(x,y_2)\}$. Define
$\FSFw(G)$ as the $\FSF$ of $(G,\hat w)$.

\begin{Conj}\label{main_conjecture}
If $\FSF_{w'}$ and $\FSF_{w''}$ are connected for some $w''>w'>0$ then $\FSFw$ is connected for every $w\in [w',w'']$. Similar statement holds for disconnectedness.
Moreover,
there exists a $\gamma\in [0,\infty]$ such that $\FSF_{w}$ has a unique component whenever $w>\gamma$, and $\FSF_{w}$ has infinitely many components whenever $w<\gamma$.
\end{Conj}

We mention that having more than one component automatically implies having infinitely many for a much wider class of transitive graphs than the ones considered here (\cite{HuNa}, \cite{Ti}, or see \cite{PT} for a short direct proof for the special product graphs that we consider here).

The simplest nontrivial example of a graph of the form $\tree^d\gpr H$ is the case when $H=K_2$ is a single edge. It is not clear what to expect: on one hand the graph may be ``too close'' to the tree to produce disconnected FSF, on the other hand one may speculate that for small enough $w$ the relatively large degree of the tree could be the reason for a similar phenomenon as in \cite{PT} and make the FSF fall apart. Pengfei Tang has shown in \cite{Ta} that the $\FSF$ is connected for the unweighted question for $\tree^d\gpr H$. (His proof was worked out for a slightly different graph, but it is mentioned in \cite{Ta} that a similar argument can be applied for $G=\tree^d\gpr K_2$.) However, his method does not give a quantitative insight into connectivity within the $\FSF$ of $G$, and it does not seem to apply to the case when weights are added to the edges. The method of \cite{PT}, which needs $H$ to be relatively large, did not give an insight into this special case either. 
We settle this question of $H=K_2$ through an enumeration, which will also enable us to bound the decay of the distance between two points (see Lemma \ref{distance_distr}).

\begin{Th}\label{always1}
For every $w>0$, $\FSF_w(\tree^d\gpr K_2)$ is connected.
\end{Th}

We verify Conjecture \ref{main_conjecture} for large enough $w$, with only the assumption that $H$ is regular, finite and connected. Furthermore, we roughly sketch how the arguments in \cite{PT} can be applied to show that the conjecture holds in a neighborhood of 0 for a large class of $\tree^d$ and $H$, as in the next theorem.

\begin{Th}\label{large_small}
For every finite connected regular graph $H$ and $d$-regular tree $\tree^d$, there exists a $W<\infty$ such that 
for every $w>W$ the forest
$\FSFw(\tree^d\gpr H)$ is connected. Conversely, 
if $H$ is transitive, $d$ is large enough compared to the degree in $H$, and $|H|>d^{5/2}$, then for every $w\leq 1$ the forest
$\FSF_w(\tree^d\gpr H)$ has infinitely many components almost surely.
\end{Th}

\subsection{Notation}
Denote by $t(G)$ the (weighted) number of spanning trees of a finite graph $G$. Let $T \gpr wH$ be the graph obtained by the Cartesian product of $T$ and $H$, with $w$ weights on the edges of the form $\{(x,y_1),(x,y_2)\}$, and weight $1$ on the rest of the edges. For shorthand throughout Section~\ref{K_2} we use $\hat T=T \gpr wK_2$ for any graph $T$. In a graph product $T \gpr wH$ we call {\it bag} the subgraphs of the form $\{v\} \times wH$ where $v \in T$. $T_n$ will denote a ball with radius $n$ around a fix vertex $u$ in $\tree^d$. With a (convenient) slight redundancy, the $\FSF$ of $T\gpr wH$ is the same object as the $\FSF_w$ of $T\gpr H$.

We will rely on one particular consequence of Wilson's algorithm on finite graphs \cite{Wi}, namely, that for a finite connected graph $G$, the path between points $a,b\in V(G)$ within $\UST(G)$ has the same distribution as 
the {\it loop-erased random walk} path
$\text{LERW}_G(a\to b)$ from $a$ to $b$, which is constructed as follows. Run random walk in $G$ starting from $a$ until hitting $b$, and erase all the loops in the order of their appearence along the walk, to obtain a simple path from $a$ to $b$. The same link between the $\UST$ and LERW is true when $G$ has positive edge weights, in which case random walk on this network is understood instead of simple random walk, with weights being the conductances. See Chapter 4.1 of \cite{LP} for more details. In general, for an arbitrary walk $(X_1,..., X_n)$, LE$(X_1,..., X_n)$ will denote its loop-erasure.

\section{Product of a tree and a weighted edge} \label{K_2}

\subsection{Recursive formulas for the number of weighted spanning trees}

Fix a constant $d \geq 3$. Most of the definitions in this section depend on $d$, but we usually will not write it as an index. 

Let $u \in \tree^d$ a fixed vertex, remember that $T_n$ is the ball around $u$ with radius $n$.

\begin{Def} \label{balanced}
Define the {\it perfect $(d-1)$-ary tree} with height $n$ recursively in the following way.
 A perfect $(d-1)$-ary tree with height $0$ is a single vertex, the root.
For $n>0$ a perfect $(d-1)$-ary tree with height $n$ has a root, and it is connected with the roots of $d-1$ pieces of perfect $(d-1)$-ary trees with height $n-1$. 

\noindent For brevity, from now on we call the perfect $(d-1)$-ary tree simply as a perfect tree and we denote the height $n$ perfect tree by $A_n$.
\end{Def}

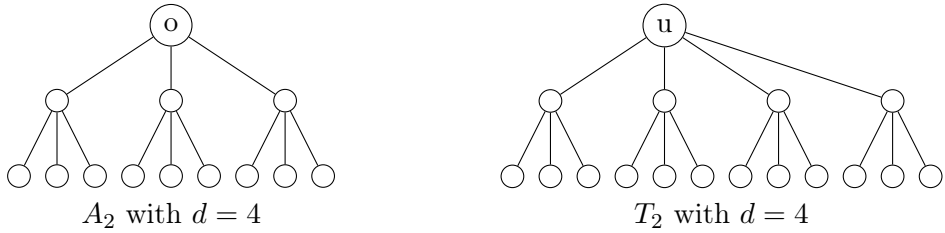
\begin{figure}[h]
\centering
\begin{minipage}{0.45\textwidth}
\centering
\begin{tikzpicture}[baseline=0,
    vert/.style={circle, draw, minimum size=7, inner sep=3}
    ]
    \draw (0, 0) node[vert] (x0l0) {o};
    
    \foreach \x/\y/\n in {-1.5/-1/x0l1, 1.5/-1/x2l1, 0/-1/x1l1,
    -2.0/-2/x0l2,-1.5/-2/x1l2,-1.0/-2/x2l2,-0.5/-2/x3l2,0.0/-2/x4l2,0.5/-2/x5l2,1.0/-2/x6l2,1.5/-2/x7l2,2.0/-2/x8l2}
    \draw (\x, \y) node[vert] (\n) {};
    \foreach \p/\q in {x0l0/x0l1, x0l0/x1l1, x0l0/x2l1,
        x0l1/x0l2,x0l1/x1l2,x0l1/x2l2,x1l1/x3l2,x1l1/x4l2,x1l1/x5l2,x2l1/x6l2,x2l1/x7l2,x2l1/x8l2}
    \draw (\p) -- (\q);
\end{tikzpicture}
    \begin{minipage}{\textwidth}
    \vspace{1ex} \centering
    $A_2$ with $d=4$
    \end{minipage}
\end{minipage}
\begin{minipage}{0.45\textwidth}
\centering
\begin{tikzpicture}[baseline=0,
    vert/.style={circle, draw, minimum size=7, inner sep=3}
    ]
    \draw (0, 0) node[vert] (x0l0) {u};
    
    \foreach \x/\y/\n in {-1.5/-1/x0l1, 1.5/-1/x2l1, 0/-1/x1l1,
    3/-1/x3l1,
    -2.0/-2/x0l2,-1.5/-2/x1l2,-1.0/-2/x2l2,-0.5/-2/x3l2,0.0/-2/x4l2,0.5/-2/x5l2,1.0/-2/x6l2,1.5/-2/x7l2,2.0/-2/x8l2,2.5/-2/x9l2,3.0/-2/x10l2,3.5/-2/x11l2
    }
    \draw (\x, \y) node[vert] (\n) {};
    \foreach \p/\q in {x0l0/x0l1, x0l0/x1l1, x0l0/x2l1, x0l0/x3l1,
    x0l1/x0l2,x0l1/x1l2,x0l1/x2l2,x1l1/x3l2,x1l1/x4l2,x1l1/x5l2,x2l1/x6l2,x2l1/x7l2,x2l1/x8l2,x3l1/x9l2,x3l1/x10l2,x3l1/x11l2
    }
    \draw (\p) -- (\q);
\end{tikzpicture}
    \begin{minipage}{\textwidth}
    \vspace{1ex} \centering
    $T_2$ with $d=4$
    \end{minipage}
\end{minipage}
\end{figure}

An alternative way to define the perfect tree is that if we delete an edge incident with $u$ from $T_n$, then the component containing $u$ is $A_n$ (and the other component is an $A_{n-1}$).


Let $o$ be the root of $A_n$. Denote by $A_{n-1}^1, A_{n-1}^2, ..., A_{n-1}^{d-1}$ the subgraphs of $A_n$ from the recursive Definition \ref{balanced}, and let their roots be $o_1, o_2, ..., o_{d-1}$, so these are the neighbours of $o$ in $A_n$. 
Let $e$ be the edge between $(o,0)$ and $(o,1)$ in $\hat A_n$.
In $\hat A_n$ and $\hat T_n$ let $(v,0)$ and $(v,1)$ be the two vertices in the bag $\{v \} \gpr wK_2$ for any vertex $v$. Let $G_i$ be the subgraph of $\hat A_n$ spanned by the vertices of $\hat A_{n-1}^i$ and $(o,0)$ and $(o,1)$ for all $1 \leq i \leq d-1$. Note that each edge of $\hat A_n$ is exactly in one of the $G_i$'s except $e$, which is contained in all of the $G_i$'s. 


\begin{Def}
For a shorthand of $t(\hat A_n)$ we use $a_n$, and let $a'_n$ denote the weighted number of spanning trees of $\hat A_n$ containing $e$. 
\end{Def}

We prove recursive formulas for these quantities.

\begin{Lemma} \label{recursion1}
\begin{equation*} \label{t'}
    a'_{n} = w\left(2a_{n-1} + \frac{1}{w} a'_{n-1}\right)^{d-1}
\end{equation*}
\end{Lemma}

\begin{figure}[h]
    \label{a'_n figure}
    \centering
    
\begin{minipage}{0.54\textwidth}
\centering
\begin{tikzpicture}[baseline=0,
    vert/.style={circle, draw, minimum size=7, inner sep=3},
    scale=1.5
    ]
    
    \draw (-1.8, 0.1) node[] {$\hat{A}_2$};
    \draw (-1.2, -0.9) node[] {$\hat{A}^1_1$};
    
    \foreach \x/\y/\n in {-0.3/0/l0x0l,+0.3/0/l0x0r,
        -0.8/-1/l1x0l,-0.2/-1/l1x0r,
        -1.3/-2/l2x0l,-0.7/-2/l2x0r,
        -0.0/-2/l2x1l,+0.6/-2/l2x1r,
        +1.3/-2/l2x2l,+1.9/-2/l2x2r}
    \draw (\x, \y) node[vert] (\n) {};
    
    \draw[thick] 
        (l0x0l) -- node[above] {$e$} (l0x0r)
        ;
    \draw[thick]
        (l0x0r) -- (l1x0r)
        (l2x0l) -- (l2x0r)
        (l1x0l) -- (l2x0l)
        (l1x0r) -- (l2x0r)
        (l1x0r) -- (l2x2r)
        (l2x2l) -- (l2x2r)
        (l1x0r) -- (l2x1r)
        (l1x0l) -- (l2x1l)
        ;
    \draw[dotted, thick]
        (l1x0l) -- (l1x0r)
        (l2x1l) -- (l2x1r)
        (l0x0l) -- (l1x0l)
        (l1x0l) -- (l2x2l)
        ;
    \draw 
        (1.0, -1) node[] (l1xdl) {$...$}
        (1.6, -1) node[] (l1xdr) {$...$}
        ;
    \draw[dotted, thick]
        (l0x0l) -- (l1xdl)
        ;
    \draw[thick]
        (l0x0r) -- (l1xdr)
        ;
    \draw[dashed, thin]
        (-1.5,-2.0) -- (-0.8, -0.7) -- (-0.3,-0.7) -- (+2.3, -2.0)
        ;
\end{tikzpicture}
\begin{minipage}{0.95\textwidth}
\vspace{1ex}
One of the edges from $\{(o,0),(o_1,0)\}$ 
\newline and $\{(o,1),(o_1,1)\}$ is in $T$.
\newline
$T \cap \hat{A}^1_1$ is a spanning tree in $\hat{A}^1_1$.
\end{minipage}
\end{minipage}
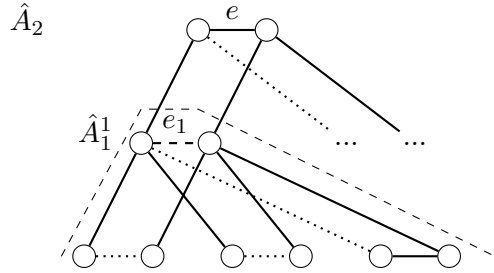
\begin{minipage}{0.45\textwidth}
\centering
\begin{tikzpicture}[baseline=0,
    vert/.style={circle, draw, minimum size=7, inner sep=3},
    scale=1.5
    ]
    \draw (-1.8, 0.1) node[] {$\hat{A}_2$};
    \draw (-1.2, -0.9) node[] {$\hat{A}^1_1$};
    
    \foreach \x/\y/\n in {-0.3/0/l0x0l,+0.3/0/l0x0r,
        -0.8/-1/l1x0l,-0.2/-1/l1x0r,
        -1.3/-2/l2x0l,-0.7/-2/l2x0r,
        -0.0/-2/l2x1l,+0.6/-2/l2x1r,
        +1.3/-2/l2x2l,+1.9/-2/l2x2r}
    \draw (\x, \y) node[vert] (\n) {};
    
    \draw[thick] 
        (l0x0l) -- node[above] {$e$} (l0x0r);
    \draw[dashed, thick]
        (l1x0l) -- node[above] {$e_1$} (l1x0r);
    \draw[thick]
        (l0x0r) -- (l1x0r)
        (l0x0l) -- (l1x0l)
        (l1x0l) -- (l2x0l)
        (l1x0r) -- (l2x0r)
        (l1x0r) -- (l2x2r)
        (l2x2l) -- (l2x2r)
        (l1x0r) -- (l2x1r)
        (l1x0l) -- (l2x1l)
        ;
    \draw[dotted, thick]
        (l2x0l) -- (l2x0r)
        (l2x1l) -- (l2x1r)
        (l1x0l) -- (l2x2l)
        ;
    \draw 
        (1.0, -1) node[] (l1xdl) {$...$}
        (1.6, -1) node[] (l1xdr) {$...$}
        ;
    \draw[dotted, thick]
        (l0x0l) -- (l1xdl)
        ;
    \draw[thick]
        (l0x0r) -- (l1xdr)
        ;
    \draw[dashed, thin]
        (-1.5,-2.0) -- (-0.8, -0.7) -- (-0.3,-0.7) -- (+2.3, -2.0)
        ;
\end{tikzpicture}

\begin{minipage}{0.95\textwidth}
\vspace{1ex}
Both edges from $\{(o,0),(o_1,0)\}$ 
\newline and $\{(o,1),(o_1,1)\}$
are in $T$.
\newline
$T \cap \hat{A}^1_1 \cup \{e_1\}$ is
a spanning tree in $\hat{A^1_1}$
\end{minipage}
\end{minipage}
\end{figure}

\begin{proof}

The $w$ multiplier comes from the weight $w$ on $e$.

Note that a subgraph $T$ of $\hat A_n$ is a spanning tree containing $e$ if and only if $T \cap G_i$ is a spanning tree containing $e$ for all $1 \leq i \leq d-1$. 

So we need to count the weighted number of spanning trees of $G_i$ containing $e$ (not multiplying with the weight $w$ on $e$ as we already counted the weight of $e$). Let $T$ be such a spanning tree. 
Consider the edges $\{(o,0),(o_i,0)\}$ and $\{(o,1),(o_i,1)\}$. This is a cut of $G_i$, so at least one of them must be in $T$. If exactly one, then we have $2$ options choosing which one, and $T \cap\hat  A^i_{n-1}$ must be a spanning tree of $\hat  A^i_{n-1}$, so there are $2a_{n-1}$ weighted options. 
If both of the edges are in $T$, then $T \cap\hat  A^i_{n-1} \cup \{\{(o_i,0),(o_i,1)\}\}$ is a spanning tree of $\hat  A^i_{n-1}$, and any spanning tree of $\hat A^i_{n-1}$ containing the edge $\{(o_i,0),(o_i,1)\}$ minus the edge $\{(o_i,0),(o_i,1)\}$ does arise as $T \cap\hat  A^i_{n-1}$, so this is bijection, and it gets a $\frac{1}{w}$ multiplier when we count the weighted number, as the weight of $\{(o_i,0),(o_i,1)\}$ is $w$. So this is $\frac{1}{w} a'_{n-1}$ weighted options. Therefore, independently for each $G_i$, we have $\left(2a_{n-1} + \frac{1}{w} a'_{n-1}\right)$ weighted possibilities. The conclusion follows.


\end{proof}

\begin{Lemma} \label{t}
\begin{equation*} 
    a_{n} = a'_{n} + (d-1)a_{n-1}\left(2a_{n-1} + \frac{1}{w} a'_{n-1}\right)^{d-2}
\end{equation*} 
\end{Lemma}

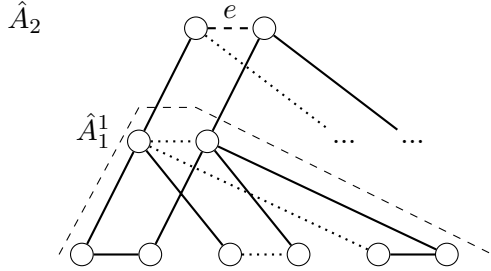
\begin{figure}[h]
\label{a_n figure}
\centering
\begin{minipage}{0.55\textwidth}
\centering
\begin{tikzpicture}[baseline=0,
    vert/.style={circle, draw, minimum size=7, inner sep=3},
    scale=1.5
    ]
    
    \draw (-1.8, 0.1) node[] {$\hat{A}_2$};
    \draw (-1.2, -0.9) node[] {$\hat{A}^1_1$};
    
    \foreach \x/\y/\n in {-0.3/0/l0x0l,+0.3/0/l0x0r,
        -0.8/-1/l1x0l,-0.2/-1/l1x0r,
        -1.3/-2/l2x0l,-0.7/-2/l2x0r,
        -0.0/-2/l2x1l,+0.6/-2/l2x1r,
        +1.3/-2/l2x2l,+1.9/-2/l2x2r}
    \draw (\x, \y) node[vert] (\n) {};
    
    \draw[dashed, thick] 
        (l0x0l) -- node[above] {$e$} (l0x0r)
        ;
    \draw[thick]
        (l0x0r) -- (l1x0r)
        (l0x0l) -- (l1x0l)
        (l2x0l) -- (l2x0r)
        (l1x0l) -- (l2x0l)
        (l1x0r) -- (l2x0r)
        (l1x0r) -- (l2x2r)
        (l2x2l) -- (l2x2r)
        (l1x0r) -- (l2x1r)
        (l1x0l) -- (l2x1l)
        ;
    \draw[dotted, thick]
        (l1x0l) -- (l1x0r)
        (l2x1l) -- (l2x1r)
        (l1x0l) -- (l2x2l)
        ;
    \draw 
        (1.0, -1) node[] (l1xdl) {$...$}
        (1.6, -1) node[] (l1xdr) {$...$}
        ;
    \draw[dotted, thick]
        (l0x0l) -- (l1xdl)
        ;
    \draw[thick]
        (l0x0r) -- (l1xdr)
        ;
    \draw[dashed, thin]
        (-1.5,-2.0) -- (-0.8, -0.7) -- (-0.3,-0.7) -- (+2.3, -2.0)
        ;
\end{tikzpicture}

\begin{minipage}{\textwidth}
\vspace{1ex}
Both edges from $\{(o,0),(o_1,0)\}$ and 
\newline $\{(o,1),(o_1,1)\}$ are in $T$, but $e$ is not.
\newline $T \cap \hat{A}^1_1$ is a spanning tree in $\hat{A}^1_1$.
\end{minipage}
\end{minipage}
\end{figure}

\begin{proof}
The weighted number of spanning trees that contain $e$ is $a'_{n}$. It is easy to see that $T$ is a spanning tree that does not contain $e$, if and only if $T \cap G_i$ is a spanning tree of $G_i$ not containing $e$ for some $i$, and $T \cap G_j$ is a graph not containing $e$, with $T \cap G_j \cup \{e\}$ is a spanning tree of $G_j$ for all $j \neq i$. We have $d-1$ options to choose $i$, then $\{(o,0),(o_i,0)\}$ and $\{(o,1),(o_i,1)\}$ must be in $T$ and $T \cap\hat  A^i_{n-1}$ is a spanning tree of $\hat  A^i_{n-1}$, so this is $a_{n-1}$ weighted options. For the other $j$'s it is exactly the same as in the proof of Lemma~\ref{recursion1}, so $\left(2a_{n-1} + \frac{1}{w} a'_{n-1}\right)$ weighted possibilities, from which the proof is complete.
\end{proof}

\begin{Lemma}
\begin{equation*} \label{t(T_n)}
    t(\hat T_n) = 2a_na_{n-1}+\frac{1}{w}(a_na'_{n-1}+a'_na_{n-1})
\end{equation*} 
\end{Lemma}

\begin{figure}[h]
    \label{T_n figure}
    \centering
    
\begin{minipage}{0.3\textwidth}
\begin{tikzpicture}[baseline=0,
    vert/.style={circle, draw, minimum size=7, inner sep=3},
    scale=1.25
    ]
    \foreach \x/\y/\n in {
        -0.3/0/x0l0l,
        +0.3/0/x0l0r,
        -1.6/-1/x0l1l,
        -1.0/-1/x0l1r,
        -0.3/-1/x1l1l,
        +0.3/-1/x1l1r,
        +1.0/-1/x2l1l,
        +1.6/-1/x2l1r,
        -0.3/1/x0l2l,
        +0.3/1/x0l2r}
    \draw (\x, \y) node[vert] (\n) {};
    
    \foreach \a/\b in {
        x0l0l/x0l0r,
        x1l1l/x1l1r,
        x0l0l/x2l1l,
        x0l0l/x0l2l}
    \draw[dotted, thick] (\a) -- (\b);
    
    \foreach \a/\b in {
        x0l2l/x0l2r,
        x0l1l/x0l1r,
        x0l0l/x0l1l,
        x0l0l/x1l1l,
        x0l0r/x1l1r,
        x0l0r/x0l1r,
        x0l0r/x2l1r,
        x2l1l/x2l1r,
        x0l0r/x0l2r}
    \draw[thick] (\a) -- (\b);
    
    \draw[dashed] (-1.5,0.5) -- (+1.5, 0.5);
    \draw (-1.0, 0.5) node[above] {$\hat{A_0}$};
    \draw (-1.0, 0.5) node[below] {$\hat{A_1}$};
\end{tikzpicture}
\begin{minipage}{12em}
\vspace{1ex}
One edge is in T,
\newline $T \cap \hat{A}_1$ is a tree,
\newline $T \cap \hat{A}_0$ is a tree.
\newline 
\end{minipage}
\end{minipage}
\begin{minipage}{0.3\textwidth}
\begin{tikzpicture}[baseline=0,
    vert/.style={circle, draw, minimum size=7, inner sep=3},
    scale=1.25
    ]
    \foreach \x/\y/\n in {
        -0.3/0/x0l0l,
        +0.3/0/x0l0r,
        -1.6/-1/x0l1l,
        -1.0/-1/x0l1r,
        -0.3/-1/x1l1l,
        +0.3/-1/x1l1r,
        +1.0/-1/x2l1l,
        +1.6/-1/x2l1r,
        -0.3/1/x0l2l,
        +0.3/1/x0l2r}
    \draw (\x, \y) node[vert] (\n) {};
    
    \foreach \a/\b in {
        x0l0l/x0l0r,
        x0l1l/x0l1r,
        x0l0l/x2l1l}
    \draw[dotted, thick] (\a) -- (\b);
    
    \foreach \a/\b in {
        x0l0r/x0l1r,
        x0l0l/x0l1l,
        x1l1l/x1l1r,
        x0l0l/x1l1l,
        x0l0r/x1l1r,
        x0l0r/x2l1r,
        x2l1l/x2l1r,
        x0l0r/x0l2r,
        x0l0l/x0l2l}
    \draw[thick] (\a) -- (\b);
    
    \draw[dashed, thick] (x0l2l) -- node[below] {$e_0$} (x0l2r);
    
    \draw[dashed] (-1.5,0.5) -- (+1.5, 0.5);
    \draw (-1.0, 0.5) node[above] {$\hat{A_0}$};
    \draw (-1.0, 0.5) node[below] {$\hat{A_1}$};
\end{tikzpicture}
\begin{minipage}{12em}
\vspace{1ex}
Both edges are in T,
\newline $T \cap \hat{A}_1$ is a tree,
\newline $T \cap \hat{A}_0 \cup \{e_0\}$ is a tree.
\newline 
\end{minipage}
\end{minipage}
\begin{minipage}{0.3\textwidth}
\begin{tikzpicture}[baseline=0,
    vert/.style={circle, draw, minimum size=7, inner sep=3},
    scale=1.25
    ]
    \foreach \x/\y/\n in {
        -0.3/0/x0l0l,
        +0.3/0/x0l0r,
        -1.6/-1/x0l1l,
        -1.0/-1/x0l1r,
        -0.3/-1/x1l1l,
        +0.3/-1/x1l1r,
        +1.0/-1/x2l1l,
        +1.6/-1/x2l1r,
        -0.3/1/x0l2l,
        +0.3/1/x0l2r}
    \draw (\x, \y) node[vert] (\n) {};
    
    \foreach \a/\b in {
        x0l1l/x0l1r,
        x1l1l/x1l1r,
        x0l0l/x2l1l}
    \draw[dotted, thick] (\a) -- (\b);
    
    \foreach \a/\b in {
        x0l0r/x0l1r,
        x0l0l/x0l1l,
        x0l0l/x1l1l,
        x0l0r/x1l1r,
        x0l0r/x2l1r,
        x2l1l/x2l1r,
        x0l0r/x0l2r,
        x0l0l/x0l2l,
        x0l2l/x0l2r}
    \draw[thick] (\a) -- (\b);
    
    \draw[dashed, thick] (x0l0l) -- node[above] {$e_1$} (x0l0r);
    
    \draw[dashed] (-1.5,0.5) -- (+1.5, 0.5);
    \draw (-1.0, 0.5) node[above] {$\hat{A_0}$};
    \draw (-1.0, 0.5) node[below] {$\hat{A_1}$};
\end{tikzpicture}
\begin{minipage}{12em}
\vspace{1ex}
Both edges are in T,
\newline $T \cap \hat{A}_1 \cup \{e_1\}$ is a tree,
\newline $T \cap \hat{A}_0$ is a tree.
\newline 
\end{minipage}
\end{minipage}
\end{figure}

\begin{proof}
$T_n$ can be constructed by taking an $A_n$ graph and an $A_{n-1}$ graph and connecting their roots. In each spanning tree $T$ of $\hat T_n$ either $T \cap \hat A_n$ or $T \cap \hat A_{n-1}$ is a spanning tree, or both. If both, then we have $2$ options to connect them, so it is $2a_na_{n-1}$ weighted options. If $\hat A_n \cap T$ is a spanning tree, but $\hat A_{n-1} \cap T$ is disconnected, then we have to put both edges between $\hat A_n$ and $\hat A_{n-1}$ into $T$, and as in Lemma \ref{recursion1}, we can think of $\hat A_{n-1} \cap T$ as a spanning tree containing the edge between the $2$ vertices in the bag of the root of $\hat A_{n-1}$, minus this edge, so it is $\frac{1}{w}a_na'_{n-1}$ weighted possibilities. In the same way we get $\frac{1}{w}a'_na_{n-1}$ for the third case. Summing these we get the desired result.    
\end{proof}

Let $t_m(\hat T_n)$ be the number of spanning trees in $\hat T_n$ with the unique path from $(u,0)$ to $(u,1)$ going into exactly $m$ bags. Note that as a bag only contains $2$ vertices, we don't have plenty of options for a path between $(u,0)$ and $(u,1)$. The only way for a path through $m$ bags is that we do $m-1$ moves in tree edges, going into the $m$'th bag, then in the $m$'th step we move within the bag, and then $m-1$ steps back up in tree edges.

\begin{Lemma} \label{m length}
For each pair $n > m \geq 2$ we have
\begin{equation*}
    t_m(\hat T_n)=d(d-1)^{m-2}w\left(2a_{n-1}+\frac{1}{w}a'_{n-1}\right)\left(2a_{n-m}+\frac{1}{w}a'_{n-m}\right) \prod_{i=1}^m \left(2a_{n-i}+\frac{1}{w}a'_{n-i}\right)^{d-2},
\end{equation*}
and in the $m=1$ case, for $n > 1$ the following is true.
$$t_1(\hat T_n)=w\left(2a_{n-1} + \frac{1}{w} a'_{n-1}\right)^d.$$
\end{Lemma}

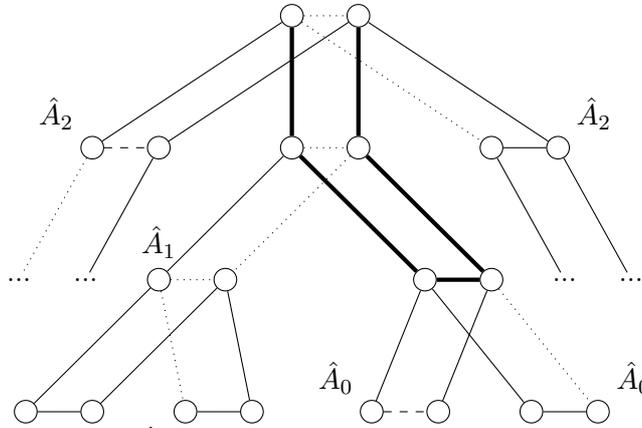
\begin{figure}[h]
    \centering
\begin{tikzpicture}[baseline=0,
    vert/.style={circle, draw, minimum size=7, inner sep=3},
    dts/.style={minimum size=5, inner sep=2},
    scale=1.75
    ]
    \foreach \x/\y/\n in {0/0/l0x0,
        -1.5/-1/l1x0,
        0/-1/l1x1,
        +1.5/-1/l1x2,
        +1.0/-2/l2x1,
        -1.0/-2/l2x0,
        -2.0/-3/l3x0,
        -0.8/-3/l3x1,
        +0.6/-3/l3x2,
        +1.8/-3/l3x3}
        \draw 
            (\x-0.25, \y) node[vert] (l\n) {}
            (\x+0.25, \y) node[vert] (r\n) {}
            ;
        
        \draw (-2.3, -2) node[dts] (ll2xd0) {...} ;
        \draw (-1.8, -2) node[dts] (rl2xd0) {...} ;
        \draw (+2.3, -2) node[dts] (rl2xd1) {...} ;
        \draw (+1.8, -2) node[dts] (ll2xd1) {...} ;
        
        \draw[ultra thick]
            (ll0x0) -- (ll1x1) -- (ll2x1) -- (rl2x1) -- (rl1x1) -- (rl0x0)
            ;
        \draw[dotted, thin]
            (ll0x0) -- (rl0x0)
            (ll1x1) -- (rl1x1)
            (ll2x1) -- (rl2x1)
            
            (ll2x0) -- (rl2x0)
            
            (ll0x0) -- (ll1x2)
            (ll2x0) -- (ll3x1)
            (rl2x1) -- (rl3x3)
            (rl1x1) -- (rl2x0)
            
            (ll1x0) -- (ll2xd0)
            ;
        \draw[]
            (ll0x0) -- (ll1x0)
            (rl0x0) -- (rl1x0)
            (rl1x0) -- (rl2xd0)
            
            (rl0x0) -- (rl1x2)
            (ll1x2) -- (rl1x2)
            (ll1x2) -- (ll2xd1)
            (rl1x2) -- (rl2xd1)
            
            (ll1x1) -- (ll2x0)
            (ll2x0) -- (ll3x0)
            (rl2x0) -- (rl3x0)
            (rl2x0) -- (rl3x1)
            (ll3x1) -- (rl3x1)
            (ll3x0) -- (rl3x0)
            
            (ll2x1) -- (ll3x2)
            (rl2x1) -- (rl3x2)
            (ll2x1) -- (ll3x3)
            (ll3x3) -- (rl3x3)
            ;
        \draw[dashed]
            (ll1x0) -- (rl1x0)
            (ll3x2) -- (rl3x2)
            ;
            
        \node[above left=0 and 0 of ll1x0] {$\hat{A}_2$} ;
        \node[above right=0 and 0 of rl1x2] {$\hat{A}_2$} ;
        \node[above=0 and 0 of ll2x0] {$\hat{A}_1$} ;
        \node[above left=0 and 0 of ll3x2] {$\hat{A}_0$} ;
        \node[above right=0 and 0 of rl3x3] {$\hat{A}_0$} ;
\end{tikzpicture}
    \begin{minipage}{0.95\textwidth}
    A possible spanning tree $T$ of $\hat T_3$ with the path highlighted between the vertices in the root bag. $2$ pieces of $\hat{A}_2$, $1$ piece of $\hat{A}_1$, and $2$ pieces $\hat{A}_0$ are hanging from the main path.
    \end{minipage}
\end{figure}

\begin{proof}
The proof of the $m=1$ case is the same as the proof of Lemma~\ref{recursion1}, except here the root has degree $d$, so the exponent is $d$ instead of $d-1$.

For the $m>1$ case there are $d(d-1)^{m-2}$ paths from $(u,0)$ to $(u,1)$ touching $m$ bags, as the first $m-1$ steps determine the path, and this is an arbitrary $m-1$ long path in $T_n$, so for the first step we have $d$ options, and for the rest there are $d-1$ possibilities. We always have exactly $1$ step within a bag in the path, which gives the multiplier $w$ in the equation.

Now assume that we know the path from $(u,0)$ to $(u,1)$.  We want to count the number of spanning trees containing this path.
The projection of this path to $T_n$ is a path of length $m-1$, with $u$ as one of its endpoints. No matter what this path is, there are always $d-1$ pieces of $A_{n-1}$, $d-2$ pieces of $A_{n-2}$, $d-2$ pieces of $A_{n-3}$, ..., $d-2$ pieces of $A_{n-m+1}$ and $d-1$ pieces of $A_{n-m}$ subtrees, that are hanging from the path, i.e. disjoint from the projected path, and with root connected to it. As in the proof of Lemma \ref{recursion1} we have independently $(2a_{n-i}+\frac{1}{w}a'_{n-i})$ weighted possibilities for each $\hat A_{n-i}$ so that the whole subgraph is a spanning tree. Multiplying these we get the number of spanning trees with this path. 
\end{proof}

\subsection{Distribution of the distances in $\FSF_w(\tree^d\gpr K_2)$}
Let $A$ be the infinite tree with degrees $d$, except one vertex, which has degree $d-1$, call this special vertex $o$. Let $e$ be the edge in the bag of $o$ in $\hat A$. Define $c:=\Prob(e \in \FSFc(A \gpr K_2))$. The $\hat A_n$ is an exhaustion of $\hat A$, so by the definition of the $\FSFc$ we have
$$c=\lim_{n \to \infty} \frac{a'_n}{a_n}.$$
Let $c_n=\frac{a'_n}{a_n}.$
\begin{Lemma} \label{s}
The sequence $s_n:=\lim_{n \to \infty} \frac{a_{n-1}^{d-1}}{a_n}$ converges to a number $s$ and
\begin{equation*} 
s=\frac{c}{w\big(2+\frac{c}{w}\big)^{d-1}}.
\end{equation*}
\end{Lemma}
\begin{proof}
From Lemma~\ref{recursion1} we have  
$$  c_na_{n} = w\left(2a_{n-1} + \frac{1}{w} a'_{n-1}\right)^{d-1}=wa_{n-1}^{d-1}\Big(2+\frac{c_{n-1}}{w}\Big)^{d-1}.$$ 
After rearranging and letting $n \to \infty$, 
$$\frac{c}{w\big(2+\frac{c}{w}\big)^{d-1}}=\lim_{n\to\infty} \frac{c_n}{w\left( 2+\frac{c_{n-1}}{w} \right)^{d-1}}=\lim_{n\to\infty} \frac{a_{n-1}^{d-1}}{a_n}=s.$$
\end{proof}

From Lemma~\ref{t},
\begin{equation*} \label{eqt_n}
   a_{n} = a'_{n} + (d-1)a_{n-1}\left(2a_{n-1} + \frac{1}{w} a'_{n-1}\right)^{d-2}=c_na_n+(d-1)a_ns_n\left(2+\frac {c_{n-1}}w\right)^{d-2}.
\end{equation*}
Dividing this by $a_n$, taking $n \rightarrow \infty$ and substituting the identity from Lemma~\ref{s}. we get
$$1=c+\frac{(d-1)c\big(2+\frac{c}{w}\big)^{d-2}}{w\big(2+\frac{c}{w}\big)^{d-1}}=c+\frac{c(d-1)}{2w+c}.$$
After rearranging we get a quadratic equation of $c$:
\begin{equation} \label{quadratic}
c^2+c(2w+d-2)-2w=0   
\end{equation}

The constant term is negative, so we have two real roots, a negative and a positive, and $c>0$ so we get the following Theorem.

\begin{Th} \label{c}
$$\Prob(e \in \text{$\FSFc$}(A \gpr K_2))=c=\frac{2-d-2w+\sqrt{(2w+d-2)^2+8w}}{2}$$
\end{Th}

\begin{proof}
From (\ref{quadratic}) with quadratic formula.
\end{proof}

\begin{Remark}
We can also get a formula for $s$ if we substitute the equation from Theorem~\ref{c} to the equation in Lemma~\ref{s}.
\end{Remark}


Let $q_m:=\lim_{n\to\infty}\frac{t_m(\hat T_n)}{t(\hat T_n)}$. This number has another meaning. This is the probability that for a $u\in\tree^d$, $(u,0)$ and $(u,1)$ belong to the same component of $\FSFc$ in $\tree^d \gpr K_2$, and their distance in the tree is $2m-1$, in other words, the path between them uses $m$ bags.
\begin{Lemma}\label{distance_distr}
For any integer $m\geq 2$, $q_m=K\left(\frac{(d-1)c}{2w+c}\right)^m$, where $K=\frac{d(2w+c)^2}{(2w+2c)(d-1)^2}$ is a constant that does not depend on $m$.
\end{Lemma}

\begin{proof}
We call two positive sequences ($u_n,\ v_n$) equivalent ($u_n\sim v_n$) if $\lim_{n\to\infty} u_n/v_n=1$.

We are going to prove, that $t_m(\hat T_n)$ and $t(\hat T_n)$ are asymptotically the same as a constant times $a_na_{n-1}$. 
$$\lim_{n\to\infty}  \frac{t_m(\hat T_n)}{t(\hat T_n)}=\lim_{n\to\infty}\frac{a_na_{n-1}}{t(\hat T_n)}\lim_{n\to\infty}\frac{t_m(\hat T_n)}{a_na_{n-1}}.$$
Recall the following constants: $$\lim_{n \to \infty} \frac{a'_n}{a_n}=c \ \ \text{and} \ \ \lim_{n \to \infty} \frac{a_{n-1}^{d-1}}{a_n}=s.$$
Using these, and Lemma~\ref{t(T_n)},
$$ \lim_{n\to\infty}\frac{a_na_{n-1}}{t(\hat T_n)}=\lim_{n\to\infty}\frac{a_na_{n-1}}{2a_na_{n-1}+\frac{1}{w}(a_na'_{n-1}+a'_na_{n-1})}=\lim_{n\to\infty}\frac{1}{2+\frac{1}{w}(\frac{a'_{n-1}}{a_{n-1}}+\frac{a'_n}{a_n})}=\frac{1}{2+\frac{2c}{w}}.$$
From Lemma~\ref{m length},
\begin{align*}
&\lim_{n\to\infty}\frac{t_m(\hat T_n)}{a_na_{n-1}}=\\ &\lim_{n\to\infty}\frac{d(d-1)^{m-2}w\Big(2a_{n-1}+\frac{1}{w}a'_{n-1}\Big)\Big(2a_{n-m}+\frac{1}{w}a'_{n-m}\Big)\prod_{i=1}^m\Big(2a_{n-i}+\frac{1}{w}a'_{n-i}\Big)^{d-2}}{a_na_{n-1}}.
\end{align*}
Here $\big(2a_{n-i}+\frac{1}{w}a'_{n-i}\big)\sim a_{n-i} \big(2+\frac cw \big)$. Hence $$t_m(\hat T_n)\sim wd(d-1)^{m-2}a_{n-1}\left(\prod_{i=1}^ma_{n-i}^{d-2}\right)a_{n-m} \Big( 2+\frac cw \Big)^{(d-2)m+2}.$$ 
Using the fact, that $a^{d-1}_{n-i}\sim sa_{n-i+1}$, one can prove easily by induction, that $a_{n-1}\left(\prod_{i=1}^ma_{n-i}^{d-2}\right)a_{n-m}\sim a_{n-1}a_ns^{m}$. Thus, combining the two calculations, $$\frac{t_m(\hat T_n)}{a_na_{n-1}}\sim wd(d-1)^{m-2}\Big(2+\frac cw\Big)^{(d-2)m+2}s^ {m}.$$
Therefore 
$$\lim_{n\to\infty}  \frac{t_m(\hat T_n)}{t(\hat T_n)}=K(d-1)^m\Big(2+\frac cw\Big)^{(d-2)m}s^{m},$$ 
where 
$$K=\frac{1}{2+\frac{2c}{w}} \cdot wd\frac{1}{(d-1)^2}\left(2+\frac{c}{w}\right)^2=\frac{d(2w+c)^2}{(2w+2c)(d-1)^2}.$$
It means, that $q_m=K\left((d-1)s(2+\frac cw)^{d-2}\right)^m.$ Using Lemma~\ref{s}, we know, that $s=\frac{c}{w\left(2+\frac{c}{w}\right)^{d-1}}$, thus $q_m=K\left(\frac{(d-1)c}{2w+c} \right)^m.$
\end{proof}

\begin{Lemma} \label{q_1}
$$\Prob(\{(u,0),(u,1)\} \in \FSFc(\tree^d \gpr K_2))=q_1=\frac{(2w+c)c}{2w+2c}.$$
\end{Lemma}

\begin{proof}
From the $m=1$ case of Lemma~\ref{m length}, we have
$$\lim_{n\to\infty} \frac{t_1(\hat T_n)}{a_na_{n-1}}=\lim_{n\to\infty} \frac{w\left(2a_{n-1} + \frac{1}{w} a'_{n-1}\right)^d}{a_na_{n-1}}=\lim_{n\to\infty} \frac{ws_na_na_{n-1}\left( 2+\frac{c_{n-1}}{w} \right)^d}{a_na_{n-1}}=ws\left( 2+\frac{c}{w} \right)^d.$$
In the proof of Lemma~\ref{distance_distr}, we calculated
$$ \lim_{n\to\infty}\frac{a_na_{n-1}}{t(\hat T_n)}=\frac{1}{2+\frac{2c}{w}}.$$
Combining these and using Lemma~\ref{s}, we have
$$q_1=\lim_{n\to\infty}  \frac{t_1(\hat T_n)}{t(\hat T_n)}=\lim_{n\to\infty}\frac{a_na_{n-1}}{t(\hat T_n)}\lim_{n\to\infty}\frac{t_1(\hat T_n)}{a_na_{n-1}}=\frac{1}{2+\frac{2c}{w}} \cdot c \left( 2+\frac{c}{w} \right)=\frac{(2w+c)c}{2w+2c}.$$
\end{proof}

Lemma~\ref{distance_distr} gives us that the distance of two vertices in the same bag in the $\FSFc$ has a geometric distribution. 

\begin{Lemma} \label{sumq}
Let $u \in \tree^d$. Then $(u,0)$ and $(u,1)$ are in the same component of the $\FSFc$ of $\tree^d \gpr K_2$ with probability $1$.
\end{Lemma}

\begin{proof}
We want to prove that the path between $(u,0)$ and $(u,1)$ is almost surely finite, so $\sum_{m=1}^{\infty} q_m=1$.
From (\ref{quadratic}) we have $2wc+c^2=2w+2c-cd$, hence
$$\frac{(2w+c)c}{2w+2c}=\frac{2w+2c-cd}{2w+2c}=1-\frac{cd}{2w+2c},$$
and (\ref{quadratic}) can also be transformed to $2w+c=\frac{2w+c-(d-1)c}{c}$, thus
$$K=\frac{d(2w+c)^2}{(2w+2c)(d-1)^2}=\frac{d((2w+c)^2-(d-1)c(2w+c))}{(2w+2c)(d-1)^2c}=\frac{cd}{2w+2c}\cdot \frac{1-\frac{(d-1)c}{2w+c}}{\left(\frac{(d-1)c}{2w+c}\right)^2}.$$
Combining these with Lemma~\ref{distance_distr} and Lemma~\ref{q_1}, we have
$$\sum_{m=1}^{\infty} q_m=\frac{(2w+c)c}{2w+2c}+K\sum_{m=2}^{\infty} \left(\frac{(d-1)c}{2w+c}\right)^m=1-\frac{cd}{2w+2c}+K\cdot\frac{\left(\frac{(d-1)c}{2w+c}\right)^2}{1-\frac{(d-1)c}{2w+c}}=1.$$
\end{proof}

\begin{Lemma}\label{elegbag}
Let $H$ be an arbitrary finite connected graph and consider the weighted graph $G=\tree^d\gpr wH$. If any pair $a,b$ of vertices in the same bag belong to the same component of $\FSFc$ almost surely, then the $\FSFc$ is almost surely connected.
\end{Lemma}
\begin{proof}
Take two adjacent bags. Let the set of edges between them called $E'$. The event that at least one edge of $E'$ is in $\FSFc$ is a cylinder event, for every graph $G_n$ of an exhausting finite sequence for $G$ this event has probability one to hold for the $\UST$. Hence in the $\FSFc$ there is an edge from $E'$ with probability one. Thus there are always two connected vertices in the two adjacent bags. By assumption, all vertices within a bag are in the same component, therefore all vertices in these two adjacent bags are in the same component. This is true for any two adjacent bags, thus for all edges in $\tree^d$. Using countable intersection, we conclude that the $\FSFc$ of that graph is connected with probability one. 
\end{proof}

Now we have everything, to prove the main result of this section:

\begin{proof}[Proof of Theorem~\ref{always1}.]
From Lemma~\ref{sumq}. and Lemma~\ref{elegbag}. the statement follows.
\end{proof}
\begin{Remark}
It is a natural question to ask whether this method can be generalized to other graphs instead of $K_2$. Unfortunately we strongly relied on the fact, that in $\hat T_n$ a path between $(u,0)$ and $(u,1)$ looks quite nice, while if we change $K_2$ to some larger graph then plenty of other options arise which we cannot handle with this enumerative method. 
\end{Remark}


\section{The general case, large and small weights}

In this section we are going to prove Theorem \ref{large_small}. The first part will follow from the next theorem.

\begin{Th}\label{ossztetel}
Given an arbitrary $d>2$ and a finite, regular, connected graph $H$, there is a $W>0$ such that the $\FSFc$ of the graph $G=\tree^d \gpr H$ is almost surely connected for all $w>W$. 
\end{Th}


As before, denote by $T_n$ the ball of radius $n$ in $\tree^d$. Let $U$ be the {\it central bag} of $T_n \gpr wH$: the bag that corresponds to the center of this ball.


\begin{Def}
A {\it trip} is a walk $(X_1,X_2,..., X_T)$ such that 
$X_i\in U$ if and only if $i=1$ or $T$.
\end{Def}

\begin{Def}\label{memorable}
Bag $D$ is {\it memorable for a trip} $(X_1,..., X_T)$, if the trip intersects $D$, and satisfies the following. If $\tau\in [1,T]$ is the last step when $X_\tau\in D$, then for every bag
$D'$($\not=U,D$) that separates $U$ and $D$, $V(D')\not=V(D')\cap \{X_{\tau+1},..., X_T\}$.
\end{Def}


\begin{Prop}\label{erased}
Let ${\cal X}=(X_1,..., X_T)$ be a walk in $T_n \gpr wH$, with $X_1$ in the central bag $U$ of $T_n \gpr wH$. Suppose that ${\cal X}'=(X_k,X_{k+1},...,X_{k'})$ is some subwalk which is a trip and intersects bag $D$. Assume that $D$ is not memorable for ${\cal X}'$. Then the loop-erasure of $(X_1,...,X_{k'})$ does not intersect $D$.
\end{Prop}

\begin{proof}
By our assumptions there exists a largest number $\tau$ with $k<\tau<k'$ and $X_\tau\in D$.
Since $D$ is not memorable, there exists a bag $D'$ that separates $U$ and $D$, and with the property that $V(D')=V(D')\cap \{X_{\tau+1},..., X_{k'}\}$. Let $t$ be the first time after $\tau$ that we enter 
$D'$. Such a $t$ exists, because $X_\tau\in D$ and $X_{k'}\in U$. Let ${\cal L}$ be LE$(X_1,..., X_{t})$. 
If ${\cal L}\cap D=\emptyset$, then the claim is proved, because we do not visit $D$ after $t>\tau$. Otherwise the first time that ${\cal L}$ enters $D'$ is strictly before $t$ (since ${\cal L}$ has to enter $D'$ before entering $D$ and reentering $D'$ at $X_t$). Let this first vertex of entrance be $v$. 
By assumption on ${\cal X}'$, $(X_t,..., X_{k'})$ visits every vertex of $D'$. Let $t'\geq t$ be the first time that $X_{t'}=v$. Then the loop-erasure of $(X_1,..., X_{t'})$ erases everything that happened after the first entrance to $D'$ at $v$. In particular, it erases every step in $D$ before $t'$, so LE$(X_1,..., X_{t'})\cap D=\emptyset$. Since $k'>t'$ (${\cal X}'$ is a trip), and no step after $t'>\tau$ is in $D$, the claim is proved.
\end{proof}

\begin{Lemma}\label{fascinationstreet}
There exists a $W>0$ such that 
for any $\alpha>0$ there is an $m$ such that the following holds for every $w>W$.
Let $u\in U$, $n>m$ and ${\cal X}=(X_1,..., X_T)$ be a trip in $T_n \gpr wH$
with 
$X_1=u$. Then
we have
$$\p({\cal X} \text{ has a memorable bag outside } T_m \gpr wH)<\alpha.$$
\end{Lemma}

\begin{proof}
Choose $\e:=\frac{1}{2(d-1)}$. Let $w>W$, where we specify $W$ at the end of this paragraph.
If $(Z_1,Z_2,...)$ is a random walk in $\tree^d \gpr wH$ started from a bag $B$, then 
let $\lambda$ be the first time when it exits $B$. If $w$ was large enough, we have for any starting vertex $x=Z_1$ and last vertex $y=Z_{\lambda}$
\begin{equation*}\label{longwalk}
\p(\{Z_1,...,Z_\lambda\}=B\mid\ Z_1=x,\ Z_{\lambda}=y
)>1-\eps,
\end{equation*}
because the minimum over $x$ and $y$ of the probability on the left tends to $1$ as $w$ goes to infinity. Fix $W$ so that the above inequality holds.

Fix bag $D$; we will use notation from Definition \ref{memorable}. Let $t$ be the first time after $\tau$
that we enter $D'$. Let $A_x$ be the event that the random walk started from a point $x\in D'$ visits every vertex of $D'$ before leaving $D'$. Denote by $\p_x$ the distribution of a random walk ${\cal Y}=(Y_1,...,Y_R)$ started from $x=Y_1$ and stopped at the first entrance to $U$. 
Let $B_x$ be the event that ${\cal Y}$ does not visit $D$. Then
\begin{equation}\label{longer}
\p ((X_{t},...,X_T)\in A_x |X_{t}=x)= \p_x (A_x| B_x)=
\p_x (A_x)\geq 1-\eps,
\end{equation}
where the last equality follows from the fact that $B_x$ is independent of $A_x$, because it only depends on the steps taken in the tree-coordinate and hence it is independent of the steps in the $H$-coordinates between two tree-coordinate steps. (To see this, note that the random walk path $(Y_1,...,Y_R)$ by $\p_x$ could be generated by first generating a suitable random walk path ${\cal T}$ in $\tree^d$, and then adding a suitably chosen random number of random $H$-steps in between every consecutive pair of steps of ${\cal T}$, independently from each other and from ${\cal T}$.)
Since $x$ was arbitrary, from \eqref{longer} we obtain 
\begin{equation}\label{maineq}
\p ((X_{t},...,X_T)\in A_{X_t})\geq 1-\eps.
\end{equation}

Let $U=D_1,D_2,..., D_\ell=D$ be the ray of bags between $U$ and $D$. Denote by $t_i$ the first time that ${\cal X}$ enters $D_i$ after $\tau$ and let $r_i$ be the first time exiting $D_i$ after $t_i$.
Finally, let $A^i$ be the the event that  $(X_{t_i},..., X_T)$ visits every vertex of $D_i$ before leaving $D_i$, in other words, $\{X_{t_i},..., X_{r_i} \}=D_i$.
Note that, conditional on $\{X_{t_i}\}$ and $\{X_{r_i}\}$, the events $\{A^j\}$ are independent, hence from the uniform lower bound \eqref{maineq} we have:
\begin{equation*}
\p \bigl(( \exists i\in[2,..., \ell-1] :\,\,  A^i  \mid   X_{t_2},...,X_{t_{l-1}},X_{r_2},...,X_{r_{l-1}}
\bigr)\geq 1-\eps^{\ell-2}.
\end{equation*}
Using the law of total probability 
\begin{equation*}\label{maineqnew}
\p \bigl(( \exists i\in[2,..., \ell-1] :\,\,  A^i 
\bigr)\geq 1-\eps^{\ell-2}.
\end{equation*}
We have just shown that $D$ is memorable for ${\cal X}$ with probability less than $\e^{\ell-2}$.
There are $d(d-1)^{s-1}$ vertices of $T_n$ with distance $s$ from the root for all $1 \leq s \leq n$, so
\begin{align*}
&\p({\cal X} \text{ has a memorable bag outside } T_m \gpr wH) \leq
 \sum_{B \not\in T_m \gpr wH} \p(B\text{ is memorable for } {\cal X} )\leq\\ &\sum_{s=m+1}^n \e^{s-2}d(d-1)^{s-1}=
\frac{d}{\e} \sum_{s=m+1}^n (\e(d-1))^{s-1}<\frac{d}{\e} (\e (d-1))^m \frac{1}{1-\e (d-1)}.
\end{align*}
By definition $\e (d-1)<1$. The number on the right hand side does not depend on $n$, thus we can choose $m$ big enough so that it is less than $\alpha$.
\end{proof}

\begin{proof}[Proof of Theorem \ref{ossztetel}]
Let $W$ be as in Lemma \ref{fascinationstreet}.

Let $\alpha_0>0$ be arbitrary.
We want to prove that if $w>W$ then for all vertices $a,b$ of $G$, there exists an $m$ with
$$\lim_{n \to \infty} \p(\text{LERW}_{T_n \gpr wH}(a \to b) \text{ leaves } T_m \gpr wH) \leq \alpha_0.$$ 
This is equivalent with the definition of connectedness of the $\FSFw$ in ${\tree^d \gpr H}$, because by Wilson's algorithm $\text{LERW}_{T_n \gpr wH}(a \to b)$ has the same distribution as the path between $a$ and $b$ in $\UST (T_n \gpr wH)$. 
By Lemma \ref{elegbag}., one may assume that $a$ and $b$ are in the same bag, $U$. 
Define $h$ as the minimum of the probability over all pairs $x\not=y\in U$ that random walk in $T_n \gpr wH$ started from $x$ hits $y$ before leaving $U$.
Let $k$ be a positive integer, chosen to satisfy $\p (\text{Geom}(h)\geq k)<\alpha_0/2$, where $\text{Geom}(h)$ denotes a geometric random variable of parameter $h$.




Choose $m$ as in Lemma \ref{fascinationstreet}, with $\alpha:=\alpha_0/2k$. Denote by ${\cal X}$ a random walk started from $a$ in $T_n \gpr wH$ and stopped when first hitting $b$. 
One can construct ${\cal X}$ as follows. Start random walk from $a$. If we hit $b$ before leaving $U$, we are finished, otherwise let $X_{s_1}$ be the last step of this walk in $U$ before first leaving $U$. Then consider the trip $(X_{s_1},..., X_{t_1})$. From the last vertex $ X_{t_1}\in U$ of this trip,
continue the random walk until either hitting $b$ or exiting $U$. The probability of the former is at least $h$; 
otherwise let $X_{s_2}$ be the last vertex in $U$ before leaving $U$, and starting from this vertex generate
the trip
$(X_{s_2},..., X_{t_2})$. Continue similarly, until at some point we hit $b$ and at that point the construction of ${\cal X}$ is finished. We see that after the end of every trip 
we had probability at least 
$h$ to hit $b$, hence the total number of trips needed is stochastically dominated by a geometric random 
variable of parameter $h$. Let $J$ be such a random variable. If $\text{LE}({\cal X})$ intersects bag $D$, then $D$ is memorable for one of the sub-trips of ${\cal X}$ by Proposition \ref{erased}. The probability that a trip has a memorable bag outside
of $T_m \gpr wH$ is less than $\alpha$ by Lemma \ref{fascinationstreet}. 
A union bound gives us 
\begin{equation*}
\p(\text{LE}({\cal X}) \text{ leaves } T_m \gpr cH) <
\p(J> k)+
k\alpha
\leq \alpha_0,
\end{equation*}
completing the proof.
\end{proof}

\begin{proof}[Proof of Theorem \ref{large_small}]
The first part of the theorem is essentially Theorem \ref{ossztetel}. 

For the second part, the case of small $w$, we have the same conditions on the graph as in the unweighted case of Theorem 1.1 in \cite{PT} and
one could repeat the arguments therein, with minor modifications which we sketch next. The result of Section 2 about random walk on the tree is obviously unchanged, while Lemma 3.1 also remains valid, with a different constant $b$, for the following reason. Replace $2k^6$ in (3.2) by $2k^6/w$. Then the entire paragraph containing (3.2) remains valid if we change every occurrence of $d$ to $wd$, and that of $2k^6$ to $2k^6/w$. The rest of the proof of Lemma 3.1 in \cite{PT} goes through without any change. The ``second ingredient'', as explained after the proof of Lemma 3.1, is based on the fact that random walk does not spend much time in a bag. The key stochastic domination results are even ``more true'' than in \cite{PT} when we have small weights on the $H$-edges, while the parts about random walk within a bag, such as Lemma 3.2,
remain unchanged. 
The rest of the proof is automatically adapted to our setting.
\end{proof}

\section*{Acknowledgement} 
\noindent
The first three authors would like to thank the Rényi REU 2020 program for undergraduate research.
The third author is partially supported by the ÚNKP-20-1 New National Excellence Program of the Ministry for Innovation and
Technology from the source of the National Research, Development and Innovation Fund. \\
The last author was partially supported by the ERC Consolidator Grant 772466 ``NOISE'', and by Icelandic Research Fund Grant 185233-051.

\bigskip

\noindent
{\bf Marcell Alexy}\\
E\"otv\"os L\'or\'and University \\
H-1117 Budapest, Pázmány Péter sétány 1/C \\
\texttt{alexy.marcell[at]hotmail.com}
\medskip
\ \\
{\bf M\'arton Borb\'enyi}\\
E\"otv\"os L\'or\'and University \\
H-1117 Budapest, Pázmány Péter sétány 1/C \\
\texttt{marton.borbenyi[at]gmail.com}
\medskip
\ \\
{\bf Andr\'as Imolay}\\
E\"otv\"os L\'or\'and University \\
H-1117 Budapest, Pázmány Péter sétány 1/C \\
\texttt{imolay.andras[at]gmail.com}
\medskip
\ \\
{\bf \'Ad\'am Tim\'ar}\\
Division of Mathematics, The Science Institute, University of Iceland\\
Dunhaga 3 IS-107 Reykjavik, Iceland\\
and\\
Alfr\'ed R\'enyi Institute of Mathematics\\
Re\'altanoda u. 13-15, Budapest 1053 Hungary\\
\texttt{madaramit[at]gmail.com}\\

\end{document}